\documentclass[12pt]{amsart}
\newcommand{\ol}{\overline}
\newcommand{\ul}{\underline}

\newtheorem{theorem}{Theorem}[section]
\newtheorem{lemma}[theorem]{Lemma}
\newtheorem{proposition}[theorem]{Proposition}

 \theoremstyle{definition}

\newtheorem{condition}[theorem]{Condition}

\theoremstyle{remark}
\newtheorem{remark}[theorem]{Remark}

\numberwithin{equation}{section}

\newcommand{\abs}[1]{\lvert#1\rvert}

\begin{document}
\setlength{\baselineskip}{1.2\baselineskip}

\title[ mixed Hessian equations]
{The Dirichlet Problem for  mixed Hessian type equations on Riemannian manifolds}

\author{Xiaojuan Chen}
\address{Faculty of Mathematics and Statistics, Hubei Key Laboratory of Applied Mathematics, Hubei University,  Wuhan 430062, P.R. China}
\email{201911110410741@stu.hubu.edu.cn}

\author{Juhua Shi}
\address{Faculty of Mathematics and Statistics, Hubei Key Laboratory of Applied Mathematics, Hubei University,  Wuhan 430062, P.R. China}
\email{20210075@hubu.edu.cn}

\author{Xiaocui Wu}
\address{Faculty of Mathematics and Statistics, Hubei Key Laboratory of Applied Mathematics, Hubei University,  Wuhan 430062, P.R. China}
\email{202021104010498@stu.hubu.edu.cn}

\author{Kang Xiao}
\address{Faculty of Mathematics and Statistics, Hubei Key Laboratory of Applied Mathematics, Hubei University,  Wuhan 430062, P.R. China}
\email{202021104010499@stu.hubu.edu.cn}

\thanks{Research of the authors was supported by the National Natural Science Foundation of China No.11971157, 12101206.}

\begin{abstract}
In this paper, we derive $C^2$ estimates
for a class of mixed Hessian type equations with Dirichlet boundary condition, and obtain the existence theorem of admissible solutions for the classical Dirichlet problem of these mixed Hessian type equations.

{\em Mathematical Subject Classification (2010):}
 Primary 35J60, Secondary
35B45.

{\em Keywords:} Dirichlet problem, a priori estimates, mixed Hessian type equations.

\end{abstract}

\maketitle
\bigskip

\section{Introduction}

\medskip
In this paper, we consider the Dirichlet problem for a class of mixed Hessian type equations with the following form
\begin{equation} \label{1.1}
\begin{cases}
\sigma_k(\nabla^2u+\chi(x,u,\nabla u)) = \sum _{l=0}^{k-1} \alpha_l(x) \sigma_l (\nabla^2u+\chi(x,u,\nabla u)), &\mbox{in}~ M,\\
u= \varphi(x), &\mbox{on}~\partial M,
\end{cases}
\end{equation}
on a Riemannian manifold $(M^n, g)$ of dimension $n\geq 3$ with smooth boundary $\partial M$, where $ 3\leq k\leq n$, $\chi(x,u,\nabla u)$ is a $(0,2)$-tensor on $\overline{M}$, $\nabla u$ and $\nabla^2 u$ are the gradient and Hessian of the function $u$, respectively.
 Note that $\sigma_k$ is a $k$-Hessian operator defined by
$$\sigma_k(W):=\sigma_k(\lambda(W)),$$
where $\lambda(W)$ are the eigenvalues  of  a $(0,2)$-tensor  $W$ with respect to the metric $g$.
Recall that the G{\aa}rding's cone is defined as
\begin{equation}\label{gam}
\Gamma_k  = \{ \lambda  \in \mathbb{R}^n :\sigma _i (\lambda ) > 0,\forall~ 1 \le i \le k\}.
\end{equation}
A function $u \in C^2(\overline{M})$ is called admissible if $\lambda(\nabla^2 u+\chi(x,u,\nabla u)) \in \Gamma_{k-1}$ for any $x \in M$.
Note that for fixed $x\in \overline{M}$, $z\in \mathbb{R}$ and $p\in T_x^{\ast} M$,
$$\chi(x,z,p):  T_x^{\ast} M \times  T_x^{\ast} M \longrightarrow \mathbb{R}$$
is a symmetric bilinear map.
We shall use the notation
\[
\chi^{\xi\eta}(x,\cdot, \cdot):=\chi(x,\cdot,\cdot)(\xi, \eta),\quad \forall \xi,\eta \in T_x^{\ast} M.
\]

The equation in \eqref{1.1} with $\chi=0$
\begin{equation}\label{1.1-1}
\sigma_k(\nabla^2u) = \sum _{l=0}^{k-1} \alpha_l(x) \sigma_l (\nabla^2u), \quad \mbox{in}~M
\end{equation}
is known to have attracted much research interest and have many applications.
Specially, it is Monge-Amp\`{e}re equation when $k=n$ and $\alpha_1=\cdots=\alpha_{k-1}=0$, $k$-Hessian equation when  $\alpha_1= \cdots = \alpha_{k-1}= 0$, and $(k,l)$-Hessian quotient equation when $\alpha_0= \cdots= \alpha_{l-1}= \alpha_{l+1} = \cdots= \alpha_{k-1}=0$. The corresponding Dirichlet problem  was studied extensively, see \cite{CNS84, I87, CNS85, T95, GB2014, GBH2015, GBH2016} and so on. In fact, the mixed Hessian equation \eqref{1.1-1} is motivated from the study of many important geometric problems. For example, special Lagrangian equations introduced by Harvey and Lawson \cite{HL82} can be written as the following form,
\[
\sin \theta \sum_{k=0}^{[n/2]}(-1)^k\sigma_{2k}(\nabla^2u)+\cos \theta \sum_{k=0}^{[(n-1)/2]}(-1)^k\sigma_{2k+1}(\nabla^2u)=0.
\]
 Another important example for the equation \eqref{1.1-1} was the following equation
 $$\sigma_1(\nabla^2u)+b\sigma_n(\nabla^2u)=C$$
 for some constants $b\geq 0, C>0$, arising from the study of $J$-equation on toric manifolds by Collins-Sz\'ekelyhidi  \cite{Co17}, which was raised as a conjecture by Chen \cite{Chen00} in the study of Mabuchi energy.

As an important example for the applications of the general notion of fully
nonlinear elliptic equations developed in \cite{Kr}, Krylov studied Dirichlet problem of the equation \eqref{1.1-1} in a $(k-1)$-convex domain in $\mathbb{R}^n$ with $\alpha_l>0 (0\leq l\leq k-1)$. Recently, Guan-Zhang \cite{GZ2019v2} observed that the equation \eqref{1.1-1} is equivalent to the following equation
\begin{equation}\label{fjt}
  \frac{\sigma_k}{\sigma_{k-1}}(\nabla^2u) -\sum_{l=0}^{k-2}\alpha_l(x) \frac{\sigma_l}{\sigma_{k-1}}(\nabla^2 u) =\alpha_{k-1},
\end{equation}
and the equation is elliptic and concave in $\Gamma_{k-1}$. Then they obtained the existence of $(k-1)$-admissible solution for the Dirichlet problem of the  equation \eqref{fjt} without sign requirement for $\alpha_{k-1}$. Later the corresponding in Neumann problem, prescribed  curvature problem, complex manifolds  were also discussed in \cite{CCX19, CCX21, Zhou22,Chen20,Chen19, Zhang21, Zhou21, TX-22}.

The main motivations to our study of the equation \eqref{1.1} with the dependence of $\chi$ come from many interesting geometric problems. These include the Christoffel-Minkowski problem (see \cite{GM2003}) and the Alexandrov problem of prescribed curvature measure (see \cite{GL1997}),  which are associated with the equation \eqref{1.1} on $\mathbb{S}^n$ for $\alpha_1=\cdots=\alpha_{k-1}=0$ and $\chi=uI$. Moreover, Guan-Zhang \cite{GZ2019v2}  studied
$$\sigma_k(\nabla^2u+uI) = \sum _{l=0}^{k-1} \alpha_l(x) \sigma_l (\nabla^2u+uI), \quad \mbox{on}~\mathbb{S}^n,$$
which arises in the problem of prescribed convex combination of area measures \cite{S}.
Another analogue example for the equation \eqref{1.1} with the dependence of $\chi$ include the Darboux equation, which appears in isometric embedding (see \cite{GL1994,N1953}); the Schouten tensor equation, which is connected with a natural fully nonlinear version of the Yamabe problem (see \cite{Via2000}).
A natural problem is raised whether we can consider the Dirichlet problem for the equation \eqref{1.1} with the dependence of $\chi$.

In the study of the equation \eqref{1.1}, a priori $C^2$ estimates are crucial to the existence and regularity of solutions.
Compared with the equation \eqref{fjt} in \cite{GZ2019v2}, the equation \eqref{1.1} involves a $(0,2)$-tensor $\chi$, which is more complicated. Therefore, it is
more difficult to obtain a priori estimates, and suitable constraints on $\chi$ should be needed. Recently, Guan-Jiao \cite{GBH2015,GBH2016} considered a fully nonlinear elliptic equation with the general form
$$f(\lambda(\nabla^2u+A[u]))=\psi(x, u, \nabla u)$$
on a Riemannian manifold and derived the estimates under conditions for a $(0,2)$ tensor $A[u]=A(x,u,\nabla u)$ and $\psi$ which are close to optimal. Inspired by the Guan-Jiao's work, we introduce the following conditions:
\begin{condition}\label{cond-cxw-01}
For any $x\in \overline{M}, z\in \mathbb{R}, p\in T_x \overline{M}$, $\xi\in T_xM$, $\chi$  satisfies
\begin{align}\label{1.3}
 \chi^{\xi\xi}(x,z,p)\quad\mbox{is concave in}~ p,
 \end{align}
 \begin{align}\label{1.4}
 \chi^{\xi\xi}_z \ge 0.
 \end{align}
 % and for any $\xi, \eta \in T_x\Omega$,
 %\begin{equation*}
%\begin{cases}
%p\cdot \nabla_x\chi^{\xi\xi}(x,z,p)+\abs{p}^2\chi_z^{\xi\xi}(x,z,p)\le \ol{\psi}_1(x,z)\abs{\xi}^2(1+\abs{p}^{\gamma_1}), \\
%\chi_{p_kp_l}^{\xi\xi}(x,z,p)\eta_k\eta_l\le-c_0\abs{\xi}^2\abs{\eta}^2+c_0\abs{\xi\cdot\eta}^2,
%\end{cases}
%\end{equation*}
%with some function $\ol{\psi}_1>0$ and constant $\gamma_1\in(0,4)$.
\end{condition}

Then the second order estimates for the equation \eqref{1.1} are as follows.

\begin{theorem}\label{th1.2}
Let  $3 \leq k \leq n$, $\varphi, \alpha_l$ be smooth functions with $\alpha_l>0$ for $0\leq l\leq k-2$, $u$ be a smooth admissible solution (i.e. $\lambda(\nabla^{2}u+\chi(x,u,\nabla u)\in \Gamma_{k-1}$)  for the equation \eqref{1.1}.  Assume that the $(0,2)$-tensor $\chi$ satisfies Condition
\ref{cond-cxw-01} and there exists an admissible subsolution  $\underline{u}\in C^2(\overline{M})$  satisfying
\begin{equation}\label{cond-cxw-03}
\begin{cases}
\sigma_{k}(\nabla^{2}\underline{u}+\chi(x,\underline{u},\nabla \underline{u}))\ge \sum_{l=0}^{k-1}\alpha_{l}(x)\sigma_{l}(\nabla^{2}\underline{u}+\chi(x,\ul u,\nabla \underline{u})),&\text{in }M, \\
\underline{u}=\varphi(x),&\text{on }\partial M.
\end{cases}
\end{equation}
Then there exists $C>0$ depending on $n,k,l$, $\|u\|_{C^1}$, $\|\ul{u}\|_{C^2}$, $\|\chi^{ij}\|_{C^2}$, $\|\alpha_{k-1}\|_{C^2}$, $\|\alpha_l\|_{C^2}$ and $\inf\alpha_l$ with $0\leq l \leq k-2$ such that
$$\max_{\overline{M}}|\nabla^2 u|\leq C.$$
\end{theorem}

In particular, in order to obtain the  gradient estimates for the equation \eqref{1.1}, we restrict our study in case $\chi=\chi(x, p)$ and add the following conditions:

\begin{condition}\label{cond-cxw-02}
For any $x\in \overline{M}, p\in T_x \overline{M}$, $\xi, \eta \in T_xM$, $\chi$  satisfies
 \begin{equation}\label{3.1}
\begin{cases}
p\cdot \nabla_x\chi^{\xi\xi}(x,p)\le \ol{\psi}_1(x)\abs{\xi}^2(1+\abs{p}^{\gamma_1}), \\
|\chi^{\xi\eta}(x,p)|^2\leq \ol{\psi}_2(x) |\xi| |\eta| (1+\abs{p}^{\gamma_2}),
\end{cases}
\end{equation}
with some functions $\ol{\psi}_1, \ol{\psi}_2>0$ and constants $\gamma_1, \gamma_2\in(0,2)$.
\end{condition}

Then we consider the solvability of the Dirichlet problem for the equation \eqref{1.1} on  Riemannian manifolds.

\begin{theorem}\label{main01}
Let $(M, g)$ be a Riemannian manifold
with nonnegative sectional curvature, $3 \leq k \leq n$, $\varphi, \alpha_l$ be smooth functions with $\alpha_l>0$ for $0\leq l\leq k-2$. Assume that the $(0,2)$-tensor $\chi$ satisfies $\chi=\chi(x, p)$, Condition
\ref{cond-cxw-01}, Condition \ref{cond-cxw-02} and
 there exists an admissible subsolution  $\underline{u}\in C^2(\overline{M})$  satisfying \eqref{cond-cxw-03},
then there exists an admissible solution $u\in C^{\infty} (\bar{M})$ for the equation \eqref{1.1}.
\end{theorem}
\begin{remark}
Following the idea in \cite{GBH2015}, \cite{GBH2016} and \cite{GZ2019v2},
we obtain the second order  estimates for
admissible solutions under Condition \ref{cond-cxw-01}, and establish gradient estimates under Condition \ref{cond-cxw-01}, Condition \ref{cond-cxw-02}. The sub-solution condition is critical in all steps of the a priori estimates. We emphasize that, there is no sign requirement  for $\alpha_{k-1}$ in the above theorem.
%It is not hard to find that our results in Theorem \ref{main01}  would still be valid if the Riemannian manifold $M$ was replaced by the spacelike hypersurface in the $(n+1)$-dimensional Lorentz-Minkowski space.
\end{remark}

The organization of the paper is as follows. In Section 2 we start with some
preliminaries. Our proof of the estimates heavily depends on results in Section 3 and Section 4. $C^1$ estimates are given in Section 3. In Section 4 we derive the global estimates
for the second order derivatives, and finish the proof of Theorem \ref{main01}.

\section{Preliminaries}

In this section, we give some basic notations and some basic properties of elementary symmetric functions, which could be found in
\cite{L96}, and establish some key lemmas.

\subsection{Basic properties of elementary symmetric functions}

For $\lambda=(\lambda_1, ... , \lambda_n)\in \mathbb{R}^n$,
the $k$-th elementary symmetric function is defined
by
\begin{equation*}
\sigma_k(\lambda) = \sum _{1 \le i_1 < i_2 <\cdots<i_k\leq n}\lambda_{i_1}\lambda_{i_2}\cdots\lambda_{i_k}.
\end{equation*}
We also set $\sigma_0=1$ and denote $\sigma_k(\lambda \left| i \right.)$ the $k$-th symmetric
function with $\lambda_i = 0$. Recall that the G{\aa}rding's cone is defined as \eqref{gam}.

%\begin{proposition}\label{prop2.2}
%Let $W={W_{ij}}$ be an $n\times n$ symmetric matric and $\lambda(W)=(\lambda_{1},\lambda_{2},\cdots,\lambda_{n})$ be the eigenvalues of the symmetric matric $W$. Suppose that $W={W_{ij}}$ is diagonal and $\lambda_{i}=W_{ii}$, then we have
%$$\frac{\partial{\lambda_{i}}}{\partial{W_{ii}}}=1,\quad \frac{\partial{\lambda_{k}}}{\partial{W_{ij}}}=0,\quad otherwise,$$
%$$\frac{\partial^{2} \lambda_{i}}{\partial{W_{ij}}\partial{W_{ji}}}=\frac{1}{\lambda_{i}-\lambda_{j}},\quad \mbox{for}~ i\neq j \quad and \quad \lambda_{i}\neq \lambda_{j},$$
%$$\frac{\partial^{2} \lambda_{i}}{\partial{W_{kl}}\partial{W_{pq}}}=0, \quad otherwise.$$
%\end{proposition}

The generalized Newton-MacLaurin inequality is as follows, which will be used all the time.
\begin{proposition}\label{prop2.4}
For $\lambda \in \Gamma_m$ and $m > l \geq 0$, $ r > s \geq 0$, $m \geq r$, $l \geq s$, we have
\begin{align}
\Bigg[\frac{{\sigma _m (\lambda )}/{C_n^m }}{{\sigma _l (\lambda )}/{C_n^l }}\Bigg]^{\frac{1}{m-l}}
\le \Bigg[\frac{{\sigma _r (\lambda )}/{C_n^r }}{{\sigma _s (\lambda )}/{C_n^s }}\Bigg]^{\frac{1}{r-s}}. \notag
\end{align}
\end{proposition}
\begin{proof}
See \cite{S05}.
\end{proof}

\subsection{Basic notations and some key lemmas}
In this paper, $\nabla$ denotes the Levi-Civita connection on $(M , g)$ and the curvature tensor
is defined by
$$R(X, Y )Z = - \nabla_X \nabla_Y Z + \nabla_Y \nabla_X Z + \nabla_{[X,Y]}Z.$$
Let $\{e_1,e_2,\cdots,e_n\}$ be local frames on $M$ and denote $g_{ij}=g(e_i,e_j)$, $\{g^{ij}\}=\{g_{ij}\}^{-1}$,
while the Christoffel symbols $\Gamma^k_{ij}$ and curvature coefficients are given respectively by $\nabla_{e_i}e_j=\Gamma^k_{ij}e_k$ and
$$R_{ijkl}=g(R(e_k,e_l)e_j,e_i),\quad R^i_{jkl}=g^{im}R_{mjkl}.$$
We shall write $\nabla_i=\nabla_{e_i}$, $\nabla_{ij}=\nabla_i\nabla_j-\Gamma^k_{ij}\nabla_k$, etc.
For a differentiable function $u$ defined on $M$, we usually identify $\nabla u$ with its gradient,
and use $\nabla^2 u$ to denote its Hessian which is locally given by $\nabla_{ij} u= \nabla_i(\nabla_j u)
-\Gamma^k_{ij}\nabla_k u$. We note that $\nabla_{ij} u=\nabla_{ji} u$ and
\begin{equation}\label{req1}
  \nabla_{ijk} u-\nabla_{jik} u=R^l_{kij}\nabla_lu,
\end{equation}
\begin{equation}\label{req0}
  \nabla_{ij}(\nabla_ku) = \nabla_{ijk}u + \Gamma^l_{ik}\nabla_{jl}u +\Gamma^l_{jk}\nabla_{il}u + \nabla_{\nabla_{ij}e_k}u,
\end{equation}
\begin{equation}\label{req2}
  \nabla_{ijkl}u-\nabla_{ikjl}u=R^m_{ljk}\nabla_{im}u+\nabla_iR^m_{ljk}\nabla_mu,
\end{equation}
\begin{equation}\label{req3}
  \nabla_{ijkl}u-\nabla_{jikl}u=R^m_{kij}\nabla_{ml}u+R^m_{lij}\nabla_{km}u.
\end{equation}
From \eqref{req2} and \eqref{req3}, we obtain
\begin{eqnarray}\label{req4}
 \nonumber \nabla_{ijkl}u-\nabla_{klij}u&=& R^m_{ljk}\nabla_{im}u+\nabla_iR^m_{ljk}\nabla_mu+R^m_{lik}\nabla_{jm}u \\
   && +R^m_{jik}\nabla_{lm}u+R^m_{jil}\nabla_{km}u+\nabla_kR^m_{jil}\nabla_mu.
\end{eqnarray}
For convenience, we introduce the following notations
\[
U:=\nabla^2u+\chi(x,u,\nabla u),~U_{ij}:=\nabla_{ij}u+\chi^{ij}(x,u, \nabla u),
\]
\[
\ul{U}:=\nabla^2\ul{u}+\chi(x,\ul u,\nabla \ul u),~\ul{U}_{ij}:=\nabla_{ij}\ul{u}+\chi^{ij}(x,\ul u,\nabla\ul u),
\]
\[
G_k(U):= \frac{\sigma_k(U)}{\sigma_{k-1}(U)},\ \ G_l(U) := -\frac{\sigma_l(U)}{\sigma_{k-1}(U)},~ 0\leq l\leq k-2,
\]

\[
G(U):= G_k(U) + \sum_{l=0}^{k-2} \alpha_l(x) G_l(U),
\]
\[
G^{ij} :=\frac{\partial G}{\partial U_{ij}},~G^{ij,rs}:=\frac{\partial^2 G}{\partial U_{ij} \partial U_{rs}},~\chi^{ij}_{p_s}:=\frac{\partial \chi^{ij}}{\partial(\nabla_s u)}, ~~ 1 \leq i, j, r, s \leq n,
\]
and
\begin{equation}\label{ll}
  \mathcal{L}:=G^{ij}\nabla_{ij}+G^{ij}\chi^{ij}_{p_s}\nabla_s.
\end{equation}

Let $u\in C^{\infty}(\overline{M})$ be an admissible solution of the equation \eqref{1.1}. Under orthonormal local frames $\{e_1,\cdots,e_n\}$, then the equation \eqref{1.1} can be rewritten as the following form:
\begin{eqnarray}\label{K-eq1}
\left\{
\begin{aligned}
& G(U):=f(\lambda[U])=\alpha_{k-1}(x), && in~ M,\\
&u(x)=\varphi(x), && on~ \partial M.
\end{aligned}
\right.
\end{eqnarray}
For simplicity, we shall still write equation \eqref{1.1} in the form \eqref{K-eq1} even if $\{e_1,\cdots,e_n\}$ are not necessarily orthonormal, although more precisely it should be
$$G([\gamma^{ik}U_{kl}\gamma^{lj}])=\alpha_{k-1}(x),$$
where $\gamma^{ij}$ is the square root of $g^{ij}: \gamma^{ik}\gamma^{kj}=g^{ij}$. Whenever we differentiate the equation, it will make no difference as long as we use covariant derivatives. Under a local frame $\{e_1,\cdots,e_n\}$, we have
\begin{equation}\label{ku}
  \nabla_k U_{ij}= \nabla_{kij} u + \nabla_k \chi^{ij}(x, u, \nabla u)+ \chi_z^{ij}(x, u, \nabla u) \nabla_k u+\chi_{p_l}^{ij}(x, u, \nabla u) \nabla_{kl} u.
\end{equation}

%\eqref{1.1}
%can also be written as
%\begin{equation*}
 % Q(\nabla^2u):=G(U)=\psi.
%\end{equation*}
%Denote by
%$$Q^{ij}=\frac{\partial G}{\partial u_{ij}}, \quad Q^{ij, r
%s}=\frac{\partial^2 G}{\partial u_{ij}\partial u_{rs}}.$$

In the establishment of the a priori estimates, the following lemma will play an important role. For more details, see \cite{GZ2019v2}.

\begin{lemma}\label{lem2.5}
If $U\in C^2(\overline{M})$  with $\lambda(U)\in \Gamma_{k-1}$ and $\alpha_l(x)>0 $ with $0 \leq l \leq k-2$,  then the operator $G(U)$
%$$G(U):= \frac{ \sigma_k(U)}{\sigma_{k-1}(U)}-\sum_{l=0}^{k-2} \alpha_l(x) \frac{\sigma_l(U)}{\sigma_{k-1}(U)}=\alpha_{k-1}(x)$$
 is elliptic and concave. Moreover
\begin{eqnarray}\label{i1}
\nonumber\sum_{i=1}^n G^{ii} &\geq& \frac{n-k+1}{k} + \sum_{l=0}^{k-2} \frac{(n-k+2)(k-l-1)}{k-1}\alpha_l(x)\frac{\sigma_l\sigma_{k-2}}{\sigma_{k-1}^2}\\
&\geq&\frac{n-k+1}{k}.
\end{eqnarray}
\end{lemma}
\begin{proof}
See \cite{Chen212}.
\end{proof}

\section{a priori estimates}

\subsection{$C^1$ estimates}

In order to estimate the gradient of $(1.1)$, we use a method similar to Theorem 4.2 in \cite{GBH2016}, but there is no sign requirement  for the right hand function $\alpha_{k-1}(x)$.
\begin{theorem}\label{C1}
Let $(M, g)$ be a Riemannian manifold
with nonnegative sectional curvature, $3 \leq k \leq n$, $\varphi, \alpha_l$ be smooth functions with $\alpha_l>0$ for $0\leq l\leq k-2$. Assume that  the $(0,2)$-tensor $\chi=\chi(x, p)$ satisfies Condition \ref{cond-cxw-01} and Condition \ref{cond-cxw-02},
 there exists an admissible subsolution  $\underline{u}\in C^2(\overline{M})$  satisfying \eqref{cond-cxw-03}. Let
 $u\in C^{\infty} (\overline{M})$ be an admissible solution  for the equation \eqref{1.1}, then
 $$\max_{\overline{M}} |\nabla u| \leq C$$
 for a constant $C$ depending on $n,k,l$, $\|u\|_{C^0}$, $\|\ul{u}\|_{C^2}$, $\|\alpha_{k-1}\|_{C^1}$, $\|\alpha_l\|_{C^1}$ and $\inf\alpha_l$ with $0 \leq l \leq k-2$.
\end{theorem}
\begin{proof}
As in \cite{GBH2016}, in order to derive the $C^0$ estimates and $C^1$ estimates, we need to restrict $\chi=\chi(x,p)$. Since $\chi$ and $\alpha_{k-1}$  are assumed to be independent of $u$, by the comparison principle, it is easy to obtain
$$\max_{\bar{M}} | u| + \max_{\partial M} |\nabla u| \leq C. $$
Hence, we only need to establish the interior gradient estimates.
Let $\omega=|\nabla u|$, $\psi=(u-\underline{u})+\sup_M(\underline{u}-u)+1$. Assume that $\omega\psi^{-\delta}$ achieves a positive maximum at an interior point $x_0\in M$ where $\delta\in(0,\frac{1}{2})$ is a constant. We may choose the local orthonormal frame $\{e_1,e_2,\cdots,e_n\}$ about $x_0$ such that $\nabla_{e_i}e_j=0$ at $x_0$ and $\{U_{ij}(x_0)\}$ is diagonal.

Thus the function $\log\omega-\delta\log\psi$ attains its maximum at $x_0$ for $i=1,\cdots,n$, hence at $x_0$,
\begin{equation}\label{gra1}
  \frac{\nabla_i\omega}{\omega}-\frac{\delta\nabla_i\psi}{\psi}=0,
\end{equation}
 \begin{equation}\label{gra2}
 \frac{\nabla_{ii}\omega}{\omega}+\frac{(\delta-\delta^2)|\nabla_i\psi|^2}{\psi^2}-\frac{\delta\nabla_{ii}\psi}{\psi}\leq 0.
 \end{equation}
 Since $(M,g)$ has nonnegative sectional curvature, in orthonormal local frame,
 $$R^k_{iil}\nabla_ku\nabla_lu\geq 0.$$
 Recall that $\nabla_{ij}u=\nabla_{ji}u$ and $\nabla_{ijk}u-\nabla_{jik}u=R^l_{kij}\nabla_lu$. Then
 \begin{eqnarray}\label{ome}
   \nonumber\omega\nabla_{ii}\omega &=& \sum_s\nabla_su\nabla_{iis}u+\sum_s\nabla_{is}u\nabla_{is}u-\nabla_i\omega\nabla_i\omega \\
  \nonumber &=& \sum_s(\nabla_{sii}u+R^k_{iis}\nabla_ku)\nabla_su+\sum_s(U_{is}-\chi^{is})^2-(\nabla_i\omega)^2 \\
  % \nonumber&\geq&\sum_s\nabla_su(\nabla_sU_{ii}-\nabla_s\chi^{ii})+\frac{1}{2}U_{ii}^2-\sum_s(\chi^{is})^2-(\nabla_i\omega)^2\\
   \nonumber&\geq&\sum_s\nabla_su\nabla_sU_{ii}-\sum_s\nabla_su\nabla_s\chi^{ii}-\sum_s\nabla_su\chi_{p_k}^{ii}\nabla_{sk}u\\
   &&+\frac{1}{2}U_{ii}^2-\sum_s(\chi^{is})^2-(\nabla_i\omega)^2.
 \end{eqnarray}
 Then by \eqref{ll}, \eqref{gra1}-\eqref{ome} and condition \eqref{3.1}, we get at $x_0$,
 \begin{eqnarray}\label{imp}
  \nonumber 0&\geq&\frac{G^{ii}\nabla_su\nabla_sU_{ii}}{\omega^2}-\frac{G^{ii}\nabla_su\nabla_s\chi^{ii}}{\omega^2}-\frac{G^{ii}\nabla_su\chi^{ii}_{p_k}
   \nabla_{sk}u}{\omega^2}\\
   \nonumber&&+\frac{(\delta-2\delta^2)G^{ii}|\nabla_i\psi|^2}{\psi^2}-\frac{\delta G^{ii}\nabla_{ii}\psi}{\psi}+\frac{1}{2\omega^2}G^{ii}U_{ii}^2-\frac{1}{\omega^2}\sum_sG^{ii}(\chi^{is})^2\\
  \nonumber &\geq&\frac{G^{ii}\nabla_su\nabla_sU_{ii}}{|\nabla u|^2}+\frac{\delta}{\psi}\mathcal{L}(\underline{u}-u)+CG^{ii}|\nabla_i\psi|^2+\frac{1}{2|\nabla u|^2}
   G^{ii}U_{ii}^2\\
   &&-C(|\nabla u|^{-2}+|\nabla u|^{\gamma_1-2}+|\nabla u|^{\gamma_2-2})\sum_iG^{ii}.
 \end{eqnarray}

  Next we need to deal with the term $\frac{\nabla_su}{|\nabla u|^2}G^{ii}\nabla_sU_{ii}$, and we can divide into two cases:

$\mathbf{Case~1:}$ If there is a positive constant $N$ such that $\frac{\sigma_l}{\sigma_{k-1}}\leq N,\quad l=0,\cdots,k-2$, then for some positive constant $C_0$,
\begin{eqnarray*}
  C_0\sum_iG^{ii}+\frac{\nabla_su}{|\nabla u|^2}G^{ii}\nabla_sU_{ii} &=&C_0\sum_iG^{ii}+\frac{\nabla_su}{|\nabla u|^2}\left(\nabla_s\alpha_{k-1}-\sum_{l=0}^{k-2}\nabla_s\alpha_lG_l\right) \\
  &\geq&\frac{C_0(n-k+1)}{k}-\frac{C}{|\nabla u|}\sum_{l=0}^{k-2}\frac{\sigma_l}{\sigma_{k-1}}-\frac{C}{|\nabla u|}\\
  &\geq&\frac{C_0(n-k+1)}{k}-\frac{C(N+1)}{|\nabla u|}\geq0,
\end{eqnarray*}
by choosing $|\nabla u|$ large enough.

$\mathbf{Case~2:}$ If $\frac{\sigma_l}{\sigma_{k-1}}> N$, then
\begin{equation*}
  \frac{\sigma_{k-2}}{\sigma_{k-1}}\geq\left(\frac{\sigma_l}{\sigma_{k-1}}\right)^{\frac{1}{k-1-l}}\geq N^{\frac{1}{k-1-l}}.
\end{equation*}
Hence by \eqref{i1}, for some positive constant $C_0$,
\begin{eqnarray*}
 %\nonumber to remove numbering (before each equation)
  &&C_0\sum_iG^{ii}+\frac{\nabla_su}{|\nabla u|^2}G^{ii}\nabla_sU_{ii} \\ &\geq&\frac{C_0(n-k+1)}{k}+\sum_{l=0}^{k-2}C_0C(n,k,l)\alpha_l(x)\frac{\sigma_l\sigma_{k-2}}{\sigma_{k-1}^2}+\frac{\nabla_su}{|\nabla u|^2}\left(\nabla_s\alpha_{k-1}-\sum_{l=0}^{k-2}\nabla_s\alpha_lG_l\right) \\
  &\geq&\sum_{l=0}^{k-2}\left(C_0C(n,k,l)\inf\alpha_l(x)N^{\frac{1}{k-1-l}}-\frac{C}{|\nabla u|}\right)\frac{\sigma_l}{\sigma_{k-1}}+\frac{C_0(n-k+1)}{k}-\frac{C}{|\nabla u|}\geq0,
\end{eqnarray*}
by choosing $|\nabla u|$ large enough.

So $\frac{\nabla_su}{|\nabla u|^2}G^{ii}\nabla_sU_{ii}\geq-C_0\sum_iG^{ii}$. Let $\lambda=\lambda(U)$, $\mu=\lambda(\underline{U})$ be the eigenvalues of $U$ and $\underline{U}$ respectively, and $\beta\in(0, \frac{1}{2\sqrt{n}})$ be a uniform constant such that
$$\nu_{\mu}-2\beta\mathbf{1}\in\Gamma_n, \quad\forall x\in \overline{M},$$
where $\nu_{\lambda}:=\frac{Df(\lambda)}{|Df(\lambda)|}$ is the unit normal vector to the level hypersurface $\partial\Gamma^{f(\lambda)}$ for $\lambda\in \Gamma$ and $\mathbf{1}=(1,\cdots,1)\in \mathbb{R}^n$, $\Gamma$ is a symmetric open and convex cone in $\mathbb{R}^n$ with $\Gamma_n\subset\Gamma$.

First, we consider the case $|\nu_{\mu}-\nu_{\lambda}|\geq\beta$, by Lemma 2.1 in \cite{GBH2016}, we have for some uniform constant $\varepsilon>0$,
$$G^{ii}(\underline{U}_{ii}-U_{ii})\geq G(\underline{U})-G(U)+\varepsilon(1+\sum_iG^{ii}).$$
By condition \eqref{1.3}, we get
\begin{eqnarray*}
  \chi_{p_k}^{ii}\nabla_k(\underline{u}-u)&\geq&\chi^{ii}(x,\nabla\underline{u})-\chi^{ii}(x,\nabla u),
  %&\geq& \chi^{ii}(x,\underline{u},\nabla\underline{u})-\chi^{ii}(x,u,\nabla u),
\end{eqnarray*}
then
\begin{eqnarray*}
  \mathcal{L}(\underline{u}-u) &=&G^{ij}\nabla_{ij}(\underline{u}-u)+G^{ij}\chi_{p_s}^{ij}\nabla_s(\underline{u}-u)\\
  &\geq& G^{ii}(\underline{U}_{ii}-U_{ii})\geq \varepsilon(1+\sum_iG^{ii}).
\end{eqnarray*}
Hence by \eqref{imp} and choosing $C_0\leq \inf_M\frac{\varepsilon\delta}{2\psi}$, $|\nabla u|$ large enough, we derive
\begin{eqnarray*}
  0 &\geq& \frac{\varepsilon\delta}{\psi}\sum_iG^{ii}+\frac{\varepsilon\delta}{\psi}-\left(C_0+C(|\nabla u|^{-2}+|\nabla u|^{\gamma_1-2}+|\nabla u|^{\gamma_2-2})\right)\sum_i G^{ii} \\
  &\geq&\left(\frac{\varepsilon\delta}{2\psi}-C(|\nabla u|^{-2}+|\nabla u|^{\gamma_1-2}+|\nabla u|^{\gamma_2-2})\right)\sum_i G^{ii}\\
  &\geq&\left(C_0-C(|\nabla u|^{-2}+|\nabla u|^{\gamma_1-2}+|\nabla u|^{\gamma_2-2})\right)\sum_i G^{ii},
\end{eqnarray*}
which implies $|\nabla u(x_0)|\leq C$.

Next we consider the case $|\nu_{\mu}-\nu_{\lambda}|<\beta$, then we have $G^{ii}\geq\frac{\beta}{\sqrt{n}}\sum_kG^{kk}$ for $1\leq i \leq n$ and $\mathcal{L}(\underline{u}-u)\geq 0$. Hence by \eqref{imp} and choosing $|\nabla u|$ large enough, we have
\begin{equation*}
  C|\nabla u|^4\sum_iG^{ii}\leq \left(C_0|\nabla u|^2+C(1+|\nabla u|^{\gamma_1}+|\nabla u|^{\gamma_2})\right)\sum_iG^{ii}.
\end{equation*}
Hence we derive $|\nabla u(x_0)|\leq C$ and the proof is completed.
\end{proof}

\subsection{Interior $C^{2}$ estimates}
%We follow the idea in \cite{GZ2019v2} to derive the interior $C^2$ estimates under the follow assumptions:
%\begin{align}\label{3.88}
%\alpha_{l}(x)=(g_{l}(x))^{p_{l}}\quad\text{for some}\quad0 \leq  g_{l}(x) \in C^{2}(M)
%\end{align}
%with $p_{l}\geq k-l$, $0\le l\le k-2$, and
%\begin{align}\label{3.89}
%\alpha_{k-1} \in C^{2}(M).
%\end{align}
In this section, we derive the following interior $C^2$ estimates. The treatment of this section follows from \cite{GBH2016}.
\begin{theorem}\label{theorem 3.1}
	Let $u \in C^{4}(M)\cap C^{2}(\overline{M})$ be an admissible solution of the equation \eqref{1.1}, $\varphi \in C^{\infty}(\overline {M})$,
suppose Condition \ref{cond-cxw-01} holds and there exists an admissible subsolution  $\underline{u}\in C^2(\overline{M})$  satisfying \eqref{cond-cxw-03}, then
	\begin{equation*}
			\max_M\abs{\nabla^{2}u} \leq C(1+\max_{\partial M}\abs{\nabla^{2}u}),
	\end{equation*}
	where $C$ is a constant depending on $n,k,l$, $\|u\|_{C^1}$, $\|\underline{u}\|_{C^2}$, $\|\chi^{ij}\|_{C^2}$, $\|\alpha_{k-1}\|_{C^2}$, $\|\alpha_l\|_{C^2}$ and $\inf\alpha_l$ with $0 \leq l \leq k-2$.
\end{theorem}
\begin{proof}
	Consider the auxiliary function
   $$W(x)=\max_{\xi\in T_xM,|\xi|=1}(\nabla_{\xi\xi}u+\chi^{\xi\xi}(x,u,\nabla u))e^{\phi},$$
   where $\phi=\frac{a}{2}|\nabla u|^2+b(\underline{u}-u)$ and $a\in(0,1)$, $b$ are constants to be determined later. Assume that $W(x)$ attains its maximum at an interior point $x_0\in M$, otherwise we are done. Choose a smooth orthonormal local frame $\{e_1,e_2,\cdots,e_n\}$ about $x_0$ such that $\nabla_{e_i}e_j=0$ and $U_{ij}=\nabla_{ij}u+\chi^{ij}(x,u,\nabla u)$ is diagonal.
   %For convenience, we write $u_i=\nabla_iu$, $u_{ij}=\nabla_{ij}u$, $u_{ijl}=\nabla_lu_{ij}$, $u_{ijrs}=\nabla_{rs}u_{ij}$, $\phi_i=\nabla_i\phi$, $\phi_{ij}=\nabla_{ij}\phi$.
   Suppose
   $$U_{11}(x_0)\geq\cdots\geq U_{nn}(x_0),$$
   so $W(x_0)=U_{11}(x_0)e^{\phi(x_0)}$. We define a new function $\widetilde{W}=\log U_{11}+\phi$.
   Then at $x_0$,
   \begin{equation}\label{w1}
     0=\widetilde{W}_i=\frac{\nabla_iU_{11}}{U_{11}}+\nabla_i\phi,
   \end{equation}
   \begin{equation}\label{w2}
   0\geq\widetilde{W}_{ii}=\frac{U_{11}\nabla_{ii}U_{11}-(\nabla_iU_{11})^2}{U_{11}^2}+\nabla_{ii}\phi.
   \end{equation}
   By \eqref{req1} and \eqref{req4},
   $$(\nabla_iU_{11})^2\leq (\nabla_1U_{1i})^2+CU_{11}^2,$$
   \begin{equation}\label{11ii}
     \nabla_{ii}U_{11}\geq \nabla_{11}U_{ii}-\nabla_{11}\chi^{ii}+\nabla_{ii}\chi^{11}-CU_{11}.
   \end{equation}
   Differentiating the equation \eqref{1.1} twice, then we get
   \begin{equation}\label{11}
     G^{ij}\nabla_1U_{ij}+\sum_{l=0}^{k-2}\nabla_1\alpha_lG_l=\nabla_1\alpha_{k-1},
   \end{equation}
   and
   \begin{eqnarray}\label{12}
     \nonumber&&G^{ii}\nabla_{11}U_{ii}+G_k^{ij,rs}\nabla_1U_{ij}\nabla_1U_{rs}+\sum_{l=0}^{k-2}\nabla_{11}\alpha_lG_l\\
     &&+2\sum_{l=0}^{k-2}\nabla_1\alpha_lG_l^{ij}\nabla_1U_{ij}+\sum_{l=0}^{k-2}\alpha_lG_l^{ij,rs}\nabla_1U_{ij}\nabla_1U_{rs}=\nabla_{11}\alpha_{k-1}.
   \end{eqnarray}
   By \eqref{w1} and \eqref{11},
   \begin{eqnarray}\label{Gi}
     \nonumber G^{ii}(\nabla_{ii}\chi^{11}-\nabla_{11}\chi^{ii})%&\geq&G^{ii}(\chi_{p_s}^{11}U_{iis}- \chi_{p_s}^{ii}U_{11s})\\
      %\nonumber&&+G^{ii}(\chi_{p_ip_i}^{11}U_{ii}^2-\chi_{p_1p_1}^{ii}U_{11}^2)-CU_{11}\sum_{ii}G^{ii}\\
     \nonumber&\geq&U_{11}G^{ii}\chi_{p_s}^{ii}\nabla_s\phi-CU_{11}\sum_iG^{ii}-C\sum_{i\geq2}G^{ii}U_{ii}^2\\
      &&-\sum_{l=0}^{k-2}\chi_{p_s}^{11}\nabla_s\alpha_lG_l-U_{11}^2\sum_{i\geq2}G^{ii}\chi_{p_1p_1}^{ii}.
   \end{eqnarray}
   Therefore combined with \eqref{w2}, \eqref{11ii}, \eqref{12}, \eqref{Gi} and the definition of $\mathcal{L}$,
   \begin{eqnarray}\label{0g}
     \nonumber0&\geq&G^{ii}\nabla_{ii}\phi+\frac{1}{U_{11}}G^{ii}\nabla_{ii}U_{11}-\frac{1}{U_{11}^2}G^{ii}(\nabla_iU_{11})^2\\
     \nonumber&\geq&\mathcal{L}\phi+\frac{1}{U_{11}}G^{ii}\nabla_{11}U_{ii}-C\sum_iG^{ii}-\frac{C}{U_{11}}\sum_{i\geq2}G^{ii}U_{ii}^2\\
     \nonumber&&-\frac{1}{U_{11}}\sum_{l=0}^{k-2}\chi_{p_s}^{11}\nabla_s\alpha_lG_l-U_{11}\sum_{i\geq2}G^{ii}\chi_{p_1p_1}^{ii}
     -\frac{1}{U_{11}^2}G^{ii}(\nabla_iU_{11})^2\\
     \nonumber&=&\mathcal{L}\phi-\frac{1}{U_{11}}\sum_{l=0}^{k-2}\chi_{p_s}^{11}\nabla_s\alpha_lG_l-C\sum_iG^{ii}
     -\frac{C}{U_{11}}\sum_{i\geq2}G^{ii}U_{ii}^2\\
     \nonumber&&+\frac{1}{U_{11}}[\nabla_{11}\alpha_{k-1}-G_k^{ij,rs}\nabla_1U_{ij}\nabla_1U_{rs}-\sum_{l=0}^{k-2}\nabla_{11}\alpha_lG_l-2\sum_{l=0}^{k-2}\nabla_1\alpha_lG_l^{ij}\nabla_1U_{ij}\\
     &&-\sum_{l=0}^{k-2}\alpha_lG_l^{ij,rs}\nabla_1U_{ij}\nabla_1U_{rs}]-U_{11}\sum_{i\geq2}G^{ii}\chi_{p_1p_1}^{ii}
     -\frac{1}{U_{11}^2}G^{ii}(\nabla_iU_{11})^2.
   \end{eqnarray}
   %In fact, the operator $(\frac{\sigma_{k-1}}{\sigma_l})^{\frac{1}{k-1-l}}$ is concave for $0\leq l\leq k-2$.
   %It follows that $(-\frac{1}{G_l})^{\frac{1}{k-1-l}}$ is a concave operator for $l=0, \cdots, k-2$, then
   %\begin{equation*}
 %-\frac{1}{2}G^{ij,rs}_l U_{ij1}U_{rs1}\geq -\frac{1}{2} (1+\frac{1}{k-1-l}) G^{-1}_l G^{ij}_l G^{rs}_lU_{ij1}U_{rs1}.
%\end{equation*}
  Since $(-\frac{1}{G_l})^{\frac{1}{k-1-l}}$ is a concave operator for $l=0, \cdots, k-2$, we obtain
\begin{equation}\label{fra}
\begin{aligned}
&- \frac{1}{2}\sum ^{k-2} _ {l=0} \alpha_l G^{ij,rs}_l \nabla_1U_{ij}\nabla_1U_{rs} - 2 \sum^{k-2}_{l=0} \nabla_1\alpha_l G_{l}^{ii}\nabla_1U_{ii}\\
\geq & -\sum_{l=0}^{k-2}\frac{\alpha_l}{2} (1+\frac{1}{k-1-l}) G^{-1}_l G^{ij}_l G^{rs}_l\nabla_1U_{ij}\nabla_1U_{rs} - 2 \sum^{k-2}_{l=0} \nabla_1\alpha_l G_{l}^{ii}\nabla_1U_{ii}\\
= & -\sum_{l=0}^{k-2}\frac{k-l}{2(k-1-l)}\alpha_lG_l^{-1}[G^{ij}_l \nabla_1U_{ij} + \frac{2(k-l-1)\nabla_1\alpha_l}{(k-l)\alpha_l}G_l]^2
+ \sum_{l=0}^{k-2}\frac{2(k-l-1)}{k-l}\frac{(\nabla_1\alpha_l)^2}{\alpha_l}G_l\\
\geq & \sum_{l=0}^{k-2}\frac{2(k-l-1)}{k-l}\frac{(\nabla_1\alpha_l)^2}{\alpha_l}G_l.
\end{aligned}
\end{equation}
Hence by \eqref{1.3}, \eqref{0g} and \eqref{fra},
\begin{eqnarray}\label{rs1}
  \nonumber\mathcal{L}\phi&\leq&\frac{C}{U_{11}}\sum_iG^{ii}U_{ii}^2+C\sum_iG^{ii}+\frac{C}{U_{11}}
  +\frac{1}{U_{11}}\sum_{l=0}^{k-2}\chi_{p_s}^{11}\nabla_s\alpha_lG_l\\
 \nonumber &&+\frac{1}{U_{11}^2}G^{ii}(\nabla_iU_{11})^2-\frac{1}{U_{11}}[\sum_{l=0}^{k-2}\frac{2(k-l-1)}{k-l}\frac{(\nabla_1\alpha_l)^2}{\alpha_l}G_l
  -\sum_{l=0}^{k-2}\nabla_{11}\alpha_lG_l\\
  &&-G_k^{ij,rs}\nabla_1U_{ij}\nabla_1U_{rs}-\frac{1}{2}\sum_{l=0}^{k-2}\alpha_lG_l^{ij,rs}\nabla_1U_{ij}\nabla_1U_{rs}].
\end{eqnarray}
 In order to estimate the third derivative term, we follow the idea of \cite{Ur02}.
   Let $J=\{i:3U_{ii}\leq-U_{11}\}$, as the estimates in \cite{GB2133} and \cite{GBH2016}, we have
   $$\frac{1}{U_{11}}G_k^{ij,rs}\nabla_1U_{ij}\nabla_1U_{rs}+\frac{G_k^{ii}(\nabla_iU_{11})^2}{U_{11}^2}\leq\frac{1}{U_{11}^2}\sum_{i\in J}G_k^{ii}(\nabla_iU_{11})^2
   +\frac{1}{U_{11}^2}\sum_{i\notin J}G_k^{11}(\nabla_iU_{11})^2+C\sum_{i\notin J}G_k^{ii},$$
 and similarly
 $$\frac{\alpha_l}{2U_{11}}G_l^{ij,rs}\nabla_1U_{ij}\nabla_1U_{rs}+\frac{\alpha_lG_l^{ii}(\nabla_iU_{11})^2}{U_{11}^2}\leq\frac{\alpha_l}{U_{11}^2}\sum_{i\in J}G_l^{ii}(\nabla_iU_{11})^2
   +\frac{\alpha_l}{U_{11}^2}\sum_{i\notin J}G_l^{11}(\nabla_iU_{11})^2+C\alpha_l\sum_{i\notin J}G_l^{ii}.$$
   Then combined with \eqref{w1}, we get
   \begin{eqnarray}\label{non}
     \nonumber&& \frac{G^{ii}(\nabla_iU_{11})^2}{U_{11}^2}+\frac{1}{U_{11}}G_k^{ij,rs}\nabla_1U_{ij}\nabla_1U_{rs}
     +\frac{1}{2U_{11}}\sum_{l=0}^{k-2}\alpha_lG_l^{ij,rs}\nabla_1U_{ij}\nabla_1U_{rs} \\
    \nonumber &\leq&\frac{1}{U_{11}^2}\sum_{i\in J}G^{ii}(\nabla_iU_{11})^2+\frac{1}{U_{11}^2}\sum_{i\notin J}G^{11}(\nabla_iU_{11})^2+C\sum_{i\notin J}G^{ii}\\
    \nonumber &\leq&\sum_{i\in J}G^{ii}\nabla_i\phi^2+G^{11}|\nabla\phi|^2+C\sum_{i\notin J}G^{ii}\\
     &\leq&Cb^2\sum_{i\in J}G^{ii}+Ca^2\sum_iG^{ii}U_{ii}^2+C\sum_iG^{ii}+CG^{11}(a^2U_{11}^2+b^2).
   \end{eqnarray}
   Recall that $\phi=\frac{a}{2}|\nabla u|^2+b(\underline{u}-u)$, then
   \begin{eqnarray}\label{mat}
    \nonumber \mathcal{L}\phi &=& G^{ij}\nabla_{ij}\phi+G^{ij}\chi_{p_s}^{ij}\nabla_s\phi \\
    \nonumber &\geq& \frac{a}{2}G^{ii}U_{ii}^2-Ca\sum_iG^{ii}+aG^{ii}\nabla_su\nabla_{sii}u+bG^{ii}\nabla_{ii}(\underline{u}-u)\\
     \nonumber&&+aG^{ij}\nabla_{p_s}\chi^{ij}\nabla_suU_{ss}-aG^{ii}\nabla_su\chi^{si}+bG^{ij}\nabla_{p_s}\chi^{ij}\nabla_s(\underline{u}-u)\\
     &\geq&b\mathcal{L}(\underline{u}-u)+\frac{a}{2}G^{ii}U_{ii}^2-(Ca+C)\sum_iG^{ii}-C-a\nabla_su\sum_{l=0}^{k-2}\nabla_s\alpha_lG_l.
   \end{eqnarray}
   By \eqref{rs1}-\eqref{mat}, we obtain
   \begin{eqnarray}\label{L1}
    \nonumber b\mathcal{L}(\underline{u}-u) &\leq& (\frac{C}{U_{11}}+Ca^2-\frac{a}{2})\sum_iG^{ii}U_{ii}^2+Cb^2\sum_{i\in J}G^{ii}\\
    \nonumber &&+CG^{11}(a^2U_{11}^2+b^2)+(Ca+C)\sum_iG^{ii}+C\\
    \nonumber &&+\frac{C}{U_{11}}+\frac{1}{U_{11}}\sum_{l=0}^{k-2}\chi_{p_s}^{11}\nabla_s\alpha_lG_l+a\nabla_su\sum_{l=0}^{k-2}\nabla_s\alpha_lG_l\\
     &&-\frac{1}{U_{11}}[\sum_{l=0}^{k-2}\frac{2(k-l-1)}{k-l}
   \frac{(\nabla_1\alpha_l)^2}{\alpha_l}G_l-\sum_{l=0}^{k-2}\nabla_{11}\alpha_lG_l].
   \end{eqnarray}
   Then we will study the term
   $$\frac{C}{U_{11}}+\frac{1}{U_{11}}\sum_{l=0}^{k-2}\chi_{p_s}^{11}\nabla_s\alpha_lG_l+a\nabla_su\sum_{l=0}^{k-2}\nabla_s\alpha_lG_l-\frac{1}{U_{11}}[\sum_{l=0}^{k-2}\frac{2(k-l-1)}{k-l}
   \frac{(\nabla_1\alpha_l)^2}{\alpha_l}G_l-\sum_{l=0}^{k-2}\nabla_{11}\alpha_lG_l].$$
   %The concavity of operator $\frac{\sigma_k}{\sigma_{k-1}}$ and concavity of operator $-\frac{\sigma_l}{\sigma_{k-1}}$ for any $l=0,1,\cdots, k-2$ imply $\sum_{i\in J}G^{ii}\geq \frac{n-k-1}{nk}$. For some constant $N$,

   When $\frac{\sigma_l}{\sigma_{k-1}}\leq N$, $l=0,\cdots,k-2$,
   \begin{eqnarray}\label{b2}
    \nonumber && \sum_iG^{ii}+\frac{1}{U_{11}}[\sum_{l=0}^{k-2}\frac{2(k-l-1)}{k-l}
   \frac{(\nabla_1\alpha_l)^2}{\alpha_l}G_l-\sum_{l=0}^{k-2}\nabla_{11}\alpha_lG_l]\\
   \nonumber&&-\frac{1}{U_{11}}\sum_{l=0}^{k-2}\chi_{p_s}^{11}\nabla_s\alpha_lG_l
   -\frac{C}{U_{11}}-a\nabla_su\sum_{l=0}^{k-2}\nabla_s\alpha_lG_l\\
     &\geq& \frac{n-k+1}{k}-\left(\frac{1}{U_{11}\inf\alpha_l}+a\right)CN\geq 0,
   \end{eqnarray}
   by choosing $U_{11}$ large enough and $a$ small enough.

When $\frac{\sigma_l}{\sigma_{k-1}}> N$, combined with \eqref{i1}, we have
   \begin{eqnarray}\label{b3}
    && \sum_iG^{ii}+\frac{1}{U_{11}}[\sum_{l=0}^{k-2}\frac{2(k-l-1)}{k-l}
   \frac{(\nabla_1\alpha_l)^2}{\alpha_l}G_l-\sum_{l=0}^{k-2}\nabla_{11}\alpha_lG_l]\\
   \nonumber&&-\frac{1}{U_{11}}\sum_{l=0}^{k-2}\chi_{p_s}^{11}\nabla_s\alpha_lG_l
   -\frac{C}{U_{11}}-a\nabla_su\sum_{l=0}^{k-2}\nabla_s\alpha_lG_l\\
    \nonumber &\geq& \frac{n-k+1}{k}+\frac{1}{U_{11}}[\sum_{l=0}^{k-2}\frac{2(k-l-1)}{k-l}
   \frac{(\nabla_1\alpha_l)^2}{\alpha_l}G_l-\sum_{l=0}^{k-2}\nabla_{11}\alpha_lG_l]\\
   \nonumber&&+\sum_{l=0}^{k-2}C(n,k,l)\alpha_l(x)\frac{\sigma_l\sigma_{k-2}}{\sigma_{k-1}^2}-\frac{1}{U_{11}}\sum_{l=0}^{k-2}\chi_{p_s}^{11}\nabla_s\alpha_lG_l
   -\frac{C}{U_{11}}-a\nabla_su\sum_{l=0}^{k-2}\nabla_s\alpha_lG_l\\
   \nonumber&\geq&\sum_{l=0}^{k-2}[C(n,k,l)\inf\alpha_l(x)N^{\frac{1}{k-1-l}}-(\frac{1}{U_{11}\inf\alpha_l}+a)C]\frac{\sigma_l}{\sigma_{k-1}}
    +\frac{n-k+1}{k}-\frac{C}{U_{11}}\geq0,
   \end{eqnarray}
   by choosing $U_{11}$ large enough and $a$ small enough.

   Inserting \eqref{b2} and \eqref{b3} into \eqref{L1}, then
   \begin{eqnarray}\label{un}
    \nonumber b\mathcal{L}(\underline{u}-u) &\leq& (\frac{C}{U_{11}}+Ca^2-\frac{a}{2})\sum_iG^{ii}U_{ii}^2+Cb^2\sum_{i\in J}G^{ii}\\
    &&+CG^{11}(a^2U_{11}^2+b^2)+(Ca+C+1)\sum_iG^{ii}+C.
   \end{eqnarray}
   In order to deal with \eqref{un}, we can also consider the two cases: $|\nu_{\mu}-\nu_{\lambda}|\geq\beta$ and $|\nu_{\mu}-\nu_{\lambda}|<\beta$.

   When $|\nu_{\mu}-\nu_{\lambda}|\geq\beta$, by Condition \ref{cond-cxw-01}, we have
   \begin{eqnarray*}
     \chi_{p_k}^{ii}\nabla_k(\underline{u}-u) &\geq&\chi^{ii}(x,u,\nabla\underline{u})-\chi^{ii}(x,u,\nabla u) \\
     &\geq& \chi^{ii}(x,\underline{u},\nabla\underline{u})-\chi^{ii}(x,u,\nabla u),
   \end{eqnarray*}
   then
  $$\mathcal{L}(\underline{u}-u)\geq G^{ii}(\underline{U}_{ii}-U_{ii})\geq \varepsilon(1+\sum_iG^{ii}).$$
   By \eqref{un}, when $b$ large enough, we can obtain
   $$\left(\frac{C}{U_{11}}+Ca^2-\frac{a}{2}\right)\sum_iG^{ii}U_{ii}^2+Cb^2\sum_{i\in J}G^{ii}+CG^{11}(a^2U_{11}^2+b^2)\geq 0,$$
   which implies $U_{11}(x_0)\leq C$. Otherwise the first term will be negative when $a$ small enough, and $|U_{ii}|\geq\frac{1}{3}U_{11}$ for $i\in J$.

   When $|\nu_{\mu}-\nu_{\lambda}|<\beta$, then $\nu_{\lambda}-\beta\mathbf{1}\in\Gamma_n$ and therefore
   $$G^{ii}\geq\frac{\beta}{\sqrt{n}}\sum_kG^{kk},\quad \forall~1\leq i\leq n.$$
   By \eqref{1.3}-\eqref{cond-cxw-03} and concavity of the operator $G$, we get $\mathcal{L}(\underline{u}-u)\geq0$. Denote $|\lambda|^2=\sum_i\lambda_i^2=\sum_iU_{ii}^2$, by \eqref{un} we derive
   \begin{equation}\label{star}
   \frac{\beta}{\sqrt{n}}|\lambda|^2\sum_iG^{ii}\leq\sum_iG^{ii}U_{ii}^2\leq C(1+\sum_iG^{ii}),
   \end{equation}
   by choosing $U_{11}$ large enough and $a$ small enough.

   Combined with \eqref{i1} and \eqref{star}, we have
   $$\frac{\beta}{\sqrt{n}}|\lambda|^2\sum_iG^{ii}\leq \frac{C(n+1)}{n-k+1}\sum_iG^{ii},$$
   which implies $|\lambda|\leq C$ and the proof is completed.
\end{proof}

\subsection{Second order derivatives boundary estimates}

For any fixed $x_0\in\partial M$, we can choose smooth orthonormal local frames $e_1, \cdots,e_n$ around
$x_0$ such that when restricted on $\partial M$, $e_n$ is normal to $\partial M$. For $x\in\overline{M}$,
let $\rho(x)$ and $d(x)$ denote the distances from $x$ to $x_0$ and $\partial M$ respectively, and set $M_{\delta}=\{x\in M:\rho(x)<\delta\}$.
We may assume $\rho$ and $d$ are smooth in $M_{\delta}$ by taking $\delta$ small.
Then we get the following important lemma, which plays a key role in
our boundary estimates.
\begin{lemma}\label{lem3.4}
Let $v=u-\underline{u}+td-\frac{N}{2}d^2$, $\mathcal{L}$ is defined as in \eqref{ll}, then for a positive constant $\epsilon$, there exist some uniform positive constants $t,\delta$ sufficiently small and $N$ sufficiently large such that
\begin{equation}\label{eps}
\begin{cases}
\mathcal{L}v\le -\epsilon(1+\sum_{i=1}^{n}G^{ii}),&\text{in}~M_\delta, \\
v\ge 0,&\text{on}~\partial M_{\delta}.
\end{cases}
\end{equation}
\end{lemma}
\begin{proof}
The proof is similar to lemma 4.3 in \cite{GZ2019v2}. Although the operator $\mathcal{L}$ is more complex than $G^{ij}$, it will make no difference since the extra term can be controlled. We can also consider $|\nu_{\mu}-\nu_{\lambda}|\geq\beta$ and $|\nu_{\mu}-\nu_{\lambda}|<\beta$ the two cases to derive \eqref{eps}
by choosing $N$ large enough and $t,\delta$ small enough.

\end{proof}

\begin{proof}[\textbf{Proof of Theorem 1.2}]
By Theorem \ref{theorem 3.1}, we only need to derive boundary estimates.

$\mathbf{Case~1:} $  Estimates of $\nabla_{\alpha\beta}u, \alpha, \beta=1,\cdots,n-1$ on $\partial M$.

Since $u-\underline{u}=0$ on $\partial M$, therefore,
$$\nabla_{\alpha\beta}(u-\underline{u})=-\nabla_{n}(u-\underline{u})B_{\alpha\beta}, \quad \mbox{on} ~\partial M,$$
where $B_{\alpha\beta} = \langle\nabla_{\alpha}e_{\beta},e_n\rangle$ denotes the second fundamental form of $\partial M$. Therefore,
\begin{eqnarray}\label{alp}
|\nabla_{\alpha\beta}u| \leq C, \quad \mbox{on} ~\partial M.
\end{eqnarray}

$\mathbf{Case~2:}$ Estimates of $ \nabla_{\alpha n}u$, $\alpha=1,\cdots,n-1$ on $\partial M$.

Consider the following barrier function
\begin{equation}\label{3.40}
	\Psi=A_1v+A_2\rho^2-A_3\underset{\beta<n}{\sum}\abs {\nabla_\beta (u-\varphi)}^2.
\end{equation}
Combined with lemma \ref{lem3.4}, we claim that
\begin{equation}\label{3.50}
\begin{cases}
\mathcal{L}(\Psi \pm \nabla_{\alpha}(u-\varphi))\le 0,&\text{in }M_\delta,\\
\Psi \pm \nabla_{\alpha}(u-\varphi)\ge 0,&\text{on }\partial M_\delta,
\end{cases}
\end{equation}
for suitable chosen positive constants $A_1, A_2, A_3$ and $\mathcal{L}, v$ are defined in lemma \ref{lem3.4}.
By \eqref{req0}, \eqref{ku} and \eqref{11}, we get
\begin{eqnarray*}
  \nonumber G^{ij}\nabla_{ij}(\nabla_{\beta}(u-\varphi))&\leq&-\sum_{l=0}^{k-2}\nabla_{\beta}\alpha_lG_l+C\sum_iG^{ii}|\lambda_i|+C\sum_iG^{ii}+C\\
  &\leq&C\sum_{l=0}^{k-2}\frac{\sigma_l}{\sigma_{k-1}}+C(1+\sum_iG^{ii}+\sum_iG^{ii}|\lambda_i|),
\end{eqnarray*}
where $\lambda=(\lambda_1,\cdots,\lambda_n)$ are the eigenvalues of $\{U_{ij}\}$.
Since
\begin{eqnarray*}
  \sum_iG^{ii}|\lambda_i|&\geq& \sum_i\left[\nabla_{ii}\left(\frac{\sigma_k}{\sigma_{k-1}}\right)U_{ii}-\sum_{l=0}^{k-2}\alpha_l\nabla_{ii}\left(\frac{\sigma_l}{\sigma_{k-1}}\right)
  U_{ii}\right]\\
  %&=&\frac{\sigma_k}{\sigma_{k-1}}+\sum_{l=0}^{k-2}\alpha_l\frac{\sigma_l}{\sigma_{k-1}}+\sum_{l=0}^{k-2}\alpha_l(k-l)\frac{\sigma_l}{\sigma_{k-1}}\\
  &\geq&\alpha_{k-1}+\varepsilon_0\sum_{l=0}^{k-2}\frac{\sigma_l}{\sigma_{k-1}},
\end{eqnarray*}
where $\varepsilon_0$ is a positive constant depending on $\inf\alpha_l$ for $0\leq l\leq k-2$.
We obtain
$$G^{ij}\nabla_{ij}(\nabla_{\beta}(u-\varphi))\leq C(1+\sum_iG^{ii}+\sum_iG^{ii}|\lambda_i|), \quad \mbox{in}~M_{\delta}.$$
Then the key is to derive
$$\mathcal{L}\Psi\leq -K(1+\sum_iG^{ii}+\sum_iG^{ii}|\lambda_i|),\quad \mbox{in}~M_{\delta}$$
for any positive constant $K$ and it is same as lemma 5.2 in \cite{GBH2015}. We can choose $A_1\gg A_2\gg A_3 \gg 1$ to get \eqref{3.50}, more details see \cite{GBH2015} and \cite{GBH2016}.

By the maximum principle, we have
\[
\Psi\pm\nabla_{\alpha}(u-\varphi)\geq 0,\quad \mbox{in}~M_\delta,
\]
and therefore
\begin{equation}\label{3.51}
|\nabla_{n\alpha}u|\leq\nabla_n\Psi+|\nabla_{n\alpha}\varphi|\leq C, \quad \mbox{on}~\partial M.
\end{equation}

$\mathbf{Case~3:}$ Estimates of $\nabla_{nn}u$ on $\partial M$.

We only need to show the uniform upper bound
$$\nabla_{nn}u(x)\leq C, \quad \forall~x\in\partial M,$$
since $\Gamma_{k-1}\subset\Gamma_1$ and the lower bound for $\nabla_{nn} u$ follows from the estimates of $\nabla_{\alpha \beta} u$ and $\nabla_{\alpha n} u$.

According to the main idea in \cite{GB2133}, which was originally due to Trudinger \cite{T95}, we will show that there
are uniform constants $c_0,\,R_0$ such that $(\lambda'(U_{\alpha\beta}(x)),R)\in \Gamma_{k-1}$ and
\begin{align}\label{3.53}
G(\lambda'(U_{\alpha\beta}(x)),R)\ge\alpha_{k-1}(x)+c_0,
\end{align}
for all $R>R_0$ and $x\in\partial M$. Here $\lambda'(U_{\alpha\beta})=(\lambda'_1,\dots,\lambda'_{n-1})$ denotes the eigenvalues of the
$(n-1)\times(n-1)$ matrix $\{U_{\alpha\beta}\}(1\le\alpha,\beta\le n-1)$. Suppose that we have found such $c_0$ and
$R_0$, by Lemma 1.2 in \cite{CNS85}, it follows from estimates \eqref{alp} and \eqref{3.51} that we can find $R_1\ge R_0$
such that, if $U_{nn}(x)>R_1$, then
\begin{align*}
G(U(x))\geq G(\lambda'(U_{\alpha\beta}(x)),U_{nn}(x))-\frac{c_0}{2}\geq\alpha_{k-1}(x)+\frac{c_0}{2},
\end{align*}
which contradicts to $G(U(x))= \alpha_{k-1}(x)$. Thus $U_{nn}(x)\leq R_1$.

In order to obtain the claim \eqref{3.53},  we only need to show that
\begin{align*}
\tilde{m}:=\underset{x\in\partial M}{\min}\biggl( \lim_{R\to+\infty}G(\lambda'(U_{\alpha\beta}(x)),R)-\alpha_{k-1}(x)\biggr) \geq c_0.
\end{align*}

Define
\begin{align*}
\tilde{c}=\underset{x\in\partial M}{\min}\biggl( \lim_{R\to+\infty}G(\lambda'(\ul U_{\alpha\beta}(x)),R)-G(\ul U(x))\biggr)>0,
\end{align*}
and
\begin{align*}
\tilde{F}[r_{\alpha\beta}]=\lim_{R\to+\infty}G(\lambda'([r_{\alpha\beta}]),R),
\end{align*}
for a symmetric $(n-1)\times (n-1)$ matrix $[r_{\alpha \beta}]$ with $(\lambda'([r_{\alpha\beta}]),R)\in \Gamma_{k-1}$.

Suppose that $\tilde{m}$ is achieved at a point $x_0\in\partial M$. Choose a local orthonormal frame around $x_0$ such that $U_{\alpha\beta}(x_0) (1\leq\alpha, \beta\leq n-1)$ is diagonal, $e_n$ is normal to $\partial M$ and assume $\nabla_{nn}u(x_0)\ge\nabla_{nn}\ul u(x_0)$.
Using the concavity, we know that
\begin{equation*}
\tilde{F}_0^{\alpha\beta}(U_{\alpha\beta}(x)-U_{\alpha\beta}(x_0))\ge \tilde{F}[U_{\alpha\beta}(x)]-\tilde{F}[U_{\alpha\beta}(x_0)]-\sum_{l=0}^{k-2} (\alpha_{k-1}(x)-\alpha_{k-1}(x_0)) \frac{\sigma_{l-1}}{\sigma_{k-2}}(U_{\alpha \beta}(x)),
\end{equation*}
where $\tilde{F}_0^{\alpha\beta}=\frac{\partial\tilde{F}}{\partial r_{\alpha\beta}}(U_{\alpha\beta}(x_0))$.
In particular, this implies,
\begin{equation}\label{3.55}
\begin{split}
&\quad\tilde{F}_0^{\alpha\beta}U_{\alpha\beta}(x)-\alpha_{k-1}(x)-\tilde{F}_0^{\alpha\beta}U_{\alpha\beta}(x_0)+\alpha_{k-1}(x_0)\ge -\overline{C}\mbox{dist}(x,x_0),\quad \mbox{on}~\partial M,
\end{split}
\end{equation}
for a constant $\overline{C}$ depending on $\|\alpha_l\|_{C^1}$ and $\|U_{\alpha \beta}\|_{L^{\infty}}$.
Hence
\begin{align}\label{3.56}
U_{\alpha\beta}=\ul U_{\alpha\beta}-\nabla_n(u-\ul u)B_{\alpha\beta}+\chi^{\alpha\beta}[u]-\chi^{\alpha\beta}[\ul u],\quad \mbox{on}~ \partial M,
\end{align}
where $\chi[u]:=\chi(x, u, \nabla u)$ and $\chi[\ul{u}]:=\chi(x, \ul u, \nabla \ul u)$.
Without loss of generality, we assume $\tilde{m}<\tilde{c}/2$. By \eqref{3.56}, we have at $x_0$,
\begin{equation}\label{3.57}
\begin{split}
\nabla_n(u-\ul u)(x_0)\tilde{F}_0^{\alpha\beta}B_{\alpha\beta} (x_0) \ge \frac{\tilde{c}}{2}+H[u(x_0)]-H[\ul u(x_0)],
\end{split}
\end{equation}
where $H[u]=\tilde{F}_0^{\alpha\beta}\chi^{\alpha\beta}[u]-\alpha_{k-1}(x)$.
Define
$$\Phi=-\eta\nabla_n(u-\ul u)+H[u]+Q,$$
with $\eta=\tilde{F}_0^{\alpha\beta}B_{\alpha\beta}(x)$ and
$$Q=\tilde{F}_0^{\alpha\beta}\ul u_{\alpha\beta}(x)-\tilde{F}_0^{\alpha\beta}U_{\alpha\beta}(x_0)+\alpha_{k-1}(x_0)+\overline{C}\mbox{dist}(x,x_0).$$
From \eqref{3.55} and \eqref{3.56} we see that $\Phi(x_0)=0$ and $\Phi \ge 0$ on $\partial M$ near $x_0$.\\
Note that
\begin{align*}
\abs{\mathcal{L}\nabla_k(u-\ul{u})}\le C(1+\sum_i G^{ii}+\sum_iG^{ii}\abs{\lambda_i}),
\end{align*}
and by \eqref{1.3}, we have
\begin{equation*}
\begin{split}
\mathcal{L}H&\le H_z[u]\mathcal{L}u+H_{p_k}[u]\mathcal{L} \nabla_k u+G^{ij}H_{p_kp_l}[u]\nabla_{ki} u \nabla_{lj} u+C(1+\sum_i G^{ii}+\sum_i G^{ii}\abs{\lambda_i})\\
&\le C(1+\sum_i G^{ii}+\sum_i G^{ii}\abs{\lambda_i}).
\end{split}
\end{equation*}
Therefore,
\begin{align*}
\mathcal{L}\Phi \le C(1+\sum_i G^{ii}+\sum_i G^{ii}\abs{\lambda_i}).
\end{align*}
Consider the function $\Psi$ defined in \eqref{3.40}, then for $A_1\gg A_2\gg A_3\gg 1$
\begin{equation*}
\begin{cases}
\mathcal{L}(\Psi + \Phi)\le 0,&\text{in }M_\delta,\\
\Psi + \Phi\ge 0,&\text{on }\partial M_\delta.
\end{cases}
\end{equation*}
By the maximum principle, $\Psi+\Phi\ge0$ in $M_\delta$. Thus
$$\nabla_n\Phi(x_0)\ge-\nabla_n\Psi(x_0)\ge -C.$$
Let $u^t=tu+(1-t)\ul u$, then we have
\begin{equation*}
\begin{split}
H[u]-H[\ul u]=(u-\ul u)\int_{0}^{1}H_z[u^t]dt+\sum_k \nabla_k (u-\ul u)\int_{0}^{1}H_{p_k}[u^t]dt.
\end{split}
\end{equation*}
Therefore,
\begin{align}\label{3.61}
H[u](x_0)-H[\ul u](x_0)=\nabla_n(u-\ul u)(x_0)\int_{0}^{1}H_{p_n}[u^t](x_0)dt,
\end{align}
and
\begin{equation*}
\nabla_nH[u](x_0)\le \nabla_{nn}(u-\ul u) (x_0 )\int_{0}^{1}H_{p_n}[u^t](x_0)dt+C,
\end{equation*}
since $H_{p_np_n}\le 0$, $\nabla_{nn}(u-\ul u)(x_0)\ge 0$ and $\nabla_n(u-\ul u)(x_0)\ge 0$. It follows that
\begin{eqnarray*}
\nabla_n\Phi (x_0)&\leq& -\eta(x_0)\nabla_{nn}(u-\ul u)(x_0)+\nabla_nH[u](x_0)+C\\
&\leq& \biggl(-\eta(x_0)+\int_{0}^{1}H_{p_n}[u^t](x_0)dt\biggr)\nabla_{nn}u(x_0)+C.
\end{eqnarray*}
By \eqref{3.57} and \eqref{3.61},
\begin{align*}
\eta(x_0)-\int_{0}^{1}H_{p_n}[u^t](x_0)dt\ge\frac{\tilde{c}}{2\nabla_n(u-\ul u)(x_0)}\ge \epsilon_1\tilde{c}>0
\end{align*}
for some uniform constant $\epsilon_1>0$. This gives
\begin{align*}
\nabla_{nn}u(x_0)\le \frac{C}{\epsilon_1\tilde{c}}.
\end{align*}
Combined with  \eqref{alp} and \eqref{3.51} we know all eigenvalues of $U(x_0)$ have a priori bound,
which implies that eigenvalues of $U(x_0)$ are contained in $\Gamma_{k-1}\cap M_{\delta}$.
 On the other hand,
if the eigenvalues can not touch $\partial\Gamma_{k-1}$, then for $R>0$ large enough,
\[
\tilde{m}=G(\lambda'(U_{\alpha\beta}(x_0)),R)-\alpha_{k-1}(x_0)>0.
\]
So we need to show that $\lambda(U(x_0))$ can not touch $\partial\Gamma_{k-1}$.
Recall our equation
\[
G(U)= \frac{\sigma_k(U)}{\sigma_{k-1}(U)} - \sum_{l=0}^{k-2} \alpha_l(x)\frac{\sigma_l(U)}{\sigma_{k-1}(U)}=\alpha_{k-1}(x).
\]
For $\lambda\in\Gamma_{k-1}$, we have $\sigma_k(\lambda)\sigma_{k-2}(\lambda)\le c(n,k)\sigma_{k-1}^2(\lambda)$, which implies
\[
\frac{\sigma_k(\lambda)}{\sigma_{k-1}(\lambda)}\le c(n,k)\frac{\sigma_{k-1}(\lambda)}{\sigma_{k-2}(\lambda)}
\le \tilde{c}(n,k)\frac{\sigma_{k-1}(\lambda)}{\sigma_{k-1}^{\frac{k-2}{k-1}}(\lambda)}=\tilde{c}(n,k)\sigma_{k-1}^{\frac{1}{k-1}}(\lambda).
\]
Then,
\begin{equation}\label{eqde}
\frac{\sigma_k(\lambda)}{\sigma_{k-1}(\lambda)}\le 0,\quad\text{as}\quad\lambda\to\partial\Gamma_{k-1}.
\end{equation}
By the non-degeneracy assumption ($\alpha_l(x_0)>0, 0\leq l\leq k-2$), if $\lambda(U(x_0))\to\partial\Gamma_{k-1}$, $G(U(x_0))\to-\infty$.
This contradicts with the condition that $\alpha_{k-1}\in C^{2}(\ol M)$.
\end{proof}

\section{The Dirichlet problem}
We now turn to the existence of solutions for the Dirichlet problem \eqref{1.1}. We consider the special case $\chi=\chi(x,p)$.

\begin{proof}[\textbf{Proof of Theorem 1.4}]
By Theorem \ref{C1} and \ref{theorem 3.1}, we obtain
\begin{equation}\label{eqrer}
\|u\|_{C^2(\overline{M})} \leq C.
\end{equation}
From \eqref{eqde}, we see that the equation \eqref{1.1} becomes uniformly
elliptic for admissible solutions satisfying \eqref{eqrer}. Applying Evans-Krylov theorem and Schauder theory, we can obtain the $C^{2, \alpha}$ and higher order estimates for the  admissible solutions of the equation \eqref{1.1}. Theorem \ref{main01}
may be proved by using the standard continuity method.
\end{proof}


\begin{thebibliography}{50}


\bibitem{CNS84}
L. Caffarelli, L. Nirenberg, J. Spruck, Dirichlet problem for nonlinear second order
elliptic equations I, Monge-Amp\`{e}re equations, Comm. Pure Appl. Math., 37(1984), 369-402.

\bibitem{CNS85}
L. Caffarelli, L. Nirenberg, J. Spruck, Dirichlet problem for nonlinear second order
elliptic equations III, Functions of the eigenvalues of the Hessian, Acta Math.,
155(1985), 261-301.

\bibitem{Co17} T. Collins, G. Sz\'ekelyhidi, Convergence of the
$J$-flow on toric manifolds, J. Differ. Geom., 107(2017), no. 1, 47-81.



\bibitem{CCX19}
C.Q. Chen, L. Chen, X.Q. Mei, N. Xiang, The Classical Neumann Problem for a class of mixed Hessian equations,
Studies in Applied Mathematics, 2021.

\bibitem{CCX21} C.Q. Chen, L. Chen, X.Q. Mei, N. Xiang, The Neumann problem for a class of mixed complex Hessian equations,
arXiv:2003.06147, 2020.

%\bibitem{Chen21} L. Chen,
%Hessian equations of Krylov type on K\"ahler manifolds,
%Preprint, arXiv:2107.12035, 2021.


\bibitem{Chen19} L. Chen, X. Guo, Y. He,
A class of fully nonlinear equations arising in conformal geometry,
Int. Math. Res. Not., 5(2022), 3651-3676.

\bibitem{Chen20} L. Chen, A.G. Shang, Q. Tu, A class of prescribed Weingarten curvature equations in Euclidean space,
Comm. Partial Differential Equations, 46(2021), no. 7, 1326-1343.

\bibitem{Chen212}
X.J. Chen, W. Lu, Q. Tu, N. Xiang,
Pogorelov estimates for a class of fully nonlinear equations,
Nonlinear Analysis, 212(2021).

\bibitem{Chen00}
X.X. Chen, On the lower bound of the Mabuchi energy and its application, Int. Math. Res. Not., 12(2000), 607-623.




\bibitem{GB2014}
B. Guan, Second order estimates and regularity for fully nonlinear ellitpic equations on Riemannian
manifolds, Duke Math. J., 163(2014), 1491-1524.


\bibitem{GB2133}
B. Guan, The Dirichlet problem for fully nonlinear elliptic equations on Riemannian manifolds,
arXiv:1403.2133v2, 2014.

%\bibitem{GB2012}
%B.Guan.Second order estimates and regularity for fully nonlinear elliptic equations on Riemannian
%manifolds.arXiv:1211.0181v1

\bibitem{GBH2015}
B. Guan, H.M. Jiao, Second order estimates for Hessian type fully nonlinear elliptic equations on
Riemannian manifolds, Calc.Var., 54(2015), 2693-2712.


\bibitem{GBH2016}
B. Guan, H.M. Jiao, The Dirichlet problem for Hessian type elliptic equations on Riemannian
manifolds, Discrete And Continuous Dynamical Systems., 36(2016), 701-714.


%\bibitem{CW01}
%K.S. Chou, X.J. Wang. A variation theory of the Hessian equation. Comm. Pure Appl. Math., 54(2001), 1029-1064.

%\bibitem{FY2007}
%J.X. Fu, S.T. Yau. A Monge-Amp\`ere type equation motivated by string theorey. Comm. Anal. Geom., 15(2007), 29-76.

%\bibitem{FY2008}
%J.X. Fu, S.T. Yau. The theory of superstring with flux on non-K$\ddot{a}$hler manifolds and the complex Monge-Amp\`ere equation. J. Diff. Geom., 78(2008), 369-428.

%\bibitem{GT}
%D. Gilbarg, N. Trudinger. Elliptic Partial Differential Equations of Second Order.
%Grundlehren der Mathematischen Wissenschaften, Vol. 224. Springer-Verlag,
%Berlin-New York, 1977. x+401 pp. ISBN: 3-540-08007-4.

\bibitem{GL1994}
P.F. Guan, Y.Y. Li,
 On Weyl problem with nonnegative Gauss curvature, J. Differ. Geom., 39(1994), 331-342.



\bibitem{GL1997}
P.F. Guan, Y.Y. Li,
$C^{1,1}$ Regularity for solutions of a problem of Alexandrov, Comm. Pure Appl.
Math., 50(1997), 789-811.

\bibitem{GM2003}
P.F. Guan, X.N. Ma,
The Christoffel-Minkowski problem I, Convexity of solutions of a Hessian
equation, Invent. Math., 151(2003), 553-577.


\bibitem{GZ2019v2}
P.F. Guan, X.W. Zhang, A class of curvature type equations, Pure and Applied Math Quarterly, 17(2021), No. 3, 865-907.

\bibitem{HL82}
R. Harvey, B. Lawson, Calibrated geometries, Acta Math., 148(1982), 47-157.

%\bibitem{HMW10}
% Z.L. Hou, X.N. Ma, D.M. Wu. A second order estimate for complex Hessian equations
%on a compact K\"{a}hler manifold. Math. Res. Lett., 17 (2010), 3: 547-561.

%\bibitem{HS99}
%G. Huisken, C. Sinestrari. Convexity estimates for mean curvature flow and singularities
%of mean convex surfaces. Acta Math., 183(1999), 45-70.

\bibitem{I87}
N. Ivochkina, Solutions of the Dirichlet problem for certain equations of Monge-Amp\`{e}re type (in Russian), Mat. Sb., 128(1985), 403-415.

%\bibitem{JT15}
%F. Jiang, N.S. Trudinger. Oblique boundary value problems for augmented Hessian equations I. Bulletin of Mathematical Sciences, 8(2018), 353-411.

%\bibitem{JT16}
%F. Jiang, N.S. Trudinger. Oblique boundary value problems for augmented Hessian equations II. Nonlinear Analysis: Theory, Methods $\&$ Applications,  154(2017), 148-173.

%\bibitem{JT3}
%F.D. Jiang, N.S. Trudinger. Oblique boundary value problems for augmented Hessian equation III. Comm. Part. Diff. Equa., 44(2019), 708-748.

\bibitem{Kr}
N. V. Krylov, On the general notion of fully nonlinear second order elliptic equation, Trans. Amer. Math. Soc., 3(1995), 857-895.

%\bibitem{L94}
% S.Y. Li. On the Neumann problems for Complex Monge-Amp\`{e}re equations. Indiana Univ. Math. J., 43(1994), 1099-1122.

%\bibitem{L91}
%Y.Y. Li. Interior gradient estimates for solutions of certain fully nonlinear elliptic equations. J. Diff. Equa., 90(1991), 172-185.

%\bibitem{CSE99}
%C. Leung, S.T. Yau, E. Zaslow, From special Lagrangian to Hermitian-Yang-Mills via Fourier-Mukai transform,
%Winter School on Mirror Symmetry, Vector Bundles and Lagrangian Submanifolds, Amer. Math. Soc., 2001.

\bibitem{L96}
G. Lieberman, Second order parabolic differential equations, World Scientific, 1996.

%\bibitem{L13}
%G. Lieberman. Oblique boundary value problems for elliptic equations. World Scientific Publishing, 2013.

%\bibitem{LT86}
%G. Lieberman, N. Trudinger. Nonlinear oblique boundary value problems for nonlinear
%elliptic equations. Trans. Amer. Math. Soc., 295 (1986), 2: 509-546.

%\bibitem{LT94}
%M. Lin, N.S. Trudinger. On some inequalities for elementary symmetric functions. Bull.
%Austral. Math. Soc., 50(1994), 317-326.

%\bibitem{LTU86}
%P.L. Lions, N. Trudinger, J. Urbas. The Neumann problem for equations of
%Monge-Amp\`{e}re type. Comm. Pure Appl. Math., 39 (1986), 539-563.

%\bibitem{MQ15}
%X.N. Ma, G.H. Qiu. The Neumann Problem for Hessian Equations. Communications in Mathematical Physics,
%366(2019), 1-28.

%\bibitem{MQX16}
%X.N. Ma, G.H. Qiu, J.J. Xu. Gradient estimates on Hessian equations for Neumann
%problem. Scientia Sinica Mathematica(Chinese), 46(2016): 1-10.


\bibitem{N1953}
L. Nirenberg,
 The Weyl and Minkowski problems in differential geometry in the large, Comm. Pure Appl. Math., 6(1953), 337-394.



%\bibitem{QX16}
%G.H. Qiu, C. Xia. Classical Neumann Problems for Hessian equations and Alexandrov-Fenchel’s inequalities. International Mathematics Research Notices, rnx296, 2018.


\bibitem{S}
R. Schneider, Convex bodies: The Brunn-Minkowski theory, Cambridge University, 1993.

\bibitem{S05}
J. Spruck, Geometric aspects of the theory of fully nonlinear elliptic equations, Clay Mathematics
Proceedings, 2(2005), 283-309.

%\bibitem{T87}
%N.S. Trudinger. On degenerate fully nonlinear elliptic equations in balls. Bull. Aust. Math. Soc., 35 (1987), 299-307.

\bibitem{T95}
N.S. Trudinger, On the Dirichlet problem for Hessian equations, Acta Math.,
175(1995), 151-164.

\bibitem{TX-22}
Q. Tu, N. Xiang, The Dirichlet problem for mixed Hessian equations on Hermitian manifolds, arXiv: 2201.05030, 2022.




%\bibitem{U95}
%J. Urbas. Nonlinear oblique boundary value problems for Hessian equations in two
%dimensions. Ann. Inst Henri Poincar\`{e}-Anal. Non Lin., 12(1995), 5: 507-575.

%\bibitem{U96}
%J. Urbas. Nonlinear oblique boundary value problems for two-dimensional curvature
%equations., Adv. Diff. Equa., 1(1996), 3: 301-336.


\bibitem{Ur02}
J. Urbas, Hessian equations on compact Riemannian manifolds, Nonlinear Problems in Mathematical
Physics and Related Topics II, New York, 2002, 367-377.


%\bibitem{TAS34}
%T.C.Collins,A.Jacob,S.-T.Yau.(1,1) forms with specified Lagrangian phase:
%A priori estimates and algebraic obstructions.arXiv:1508.01934.

%\bibitem{TS24}
%T.C.Collins and S.-T.Yau.Moment maps,nonlinear PDE,and stability in mirror symmetry.
%arXiv:1811.04824.



\bibitem{Via2000}
J. Viaclovsky,  Conformal geometry, contact geometry and the calculus of variations, Duke Math. J.,
101(2000), 283-316.


\bibitem{Zhang21}
Q. Zhang, Regularity of the Dirichlet Problem for the Non-degenerate Complex Quotient Equations, Int. Math. Res. Not., 23(2021), 17673-17694.

\bibitem{Zhou21}
J. D. Zhou, A class of the non-degenerate complex quotient equations on compact K\"ahler manifolds, Comm. Pure Appl. Anal., 20(2021), 2361-2377.

\bibitem{Zhou22}
J. D. Zhou, The interior gradient estimate for a class of mixed Hessian curvature equations, J. Korean Math. Soc., 59(2022), 53-69.

\end{thebibliography}
\end{document}